\documentclass[12pt]{amsart}

\usepackage{amsthm,amsmath,amssymb}
\usepackage[shortlabels]{enumitem}
\usepackage{graphicx,color}
\newcommand{\R}{\mathbb{R}}
\newcommand{\Ss}{\mathbb{S}}
\newcommand{\Z}{\mathbb{Z}}
\newcommand{\C}{\mathbb{C}}
\newcommand{\N}{\mathbb{N}}
\newcommand{\Hh}{\mathbb{H}}

\newtheorem{theorem}{Theorem}
\newtheorem{problem}{Problem}

\newtheorem{defn}{Definition}
\newtheorem{proposition}{Proposition}
\newtheorem{cor}{Corollary}
\newtheorem{lemma}{Lemma}

\usepackage{hyperref}
\usepackage[british]{babel}

\author{Yannick Guedes Bonthonneau}
\title[Extending a result of Chen et al.]{Extending a result of Chen, Erchenko and Gogolev.}

\begin{document}

\begin{abstract}
In a recent paper \cite{Chen-Erchenko-Gogolev}, Chen, Erchenko and Gogolev have proven that if a Riemannian manifold with boundary has hyperbolic geodesic trapped set, then it can be embedded into a compact manifold whose geodesic flow is Anosov. They have to introduce some assumptions that we discuss here. We explain how some can be removed, obtaining in particular a result applicable to all reasonable 3 dimensional examples.
\end{abstract}

\maketitle

In all that follows, $(M,g)$ will denote a smooth compact Riemannian manifold with smooth boundary $\partial M$, of dimension $n$. The geodesic flow is defined on a subset of $S M \times\R$, and we will always assume that the set of points $v\in SM$ whose geodesic trajectory never encounters $\partial SM$ is a hyperbolic set. As in \cite{Chen-Erchenko-Gogolev}, we consider 
\begin{problem}\label{probA}
Can $M$ be isometrically embedded as an open set of a compact manifold $(N,g')$ without boundary, so that the geodesic flow of $N$ is Anosov?
\end{problem}

Since metrics whose geodesic flow is Anosov form the $C^2$ interior of the metrics that do not have conjugate points \cite{Ruggiero-1991}, certainly $M$ must not have any conjugate point, so we will make this assumption from now on. 

It is a standard assumption in the context of inverse problems for manifolds with boundary, that the boundary is strictly convex, i.e that the second fundamental form 
\begin{equation}\label{eq:convex-boundary}
\mathrm{II}_{\partial M} > 0. 
\end{equation}
This ensures that the Riemannian distance function for points that are close to each other on the boundary is realized by small geodesics that only touch the boundary at their endpoints. Thus, the set of points whose geodesic trajectory remains inside $M$ for all times does not meet $\partial M$. In practice, this simplifies many computations so we will assume \eqref{eq:convex-boundary} holds. 

\begin{defn}
Let $M$ be a compact manifold with non empty boundary, with hyperbolic trapped set, no conjugate points, and strictly convex boundary. We say that it is \emph{an Anosov manifold with boundary}.
\end{defn}

Let us recall the result of \cite{Chen-Erchenko-Gogolev}
\begin{theorem}[Theorem A, \cite{Chen-Erchenko-Gogolev}]\label{thm:1}
Let $M$ be an Anosov manifold with boundary, and furthermore assume that the boundary components of $M$ are topological spheres, then Problem \ref{probA} has a positive solution.
\end{theorem}
Since the only connected compact manifold in one dimension is a sphere, this theorem in particular solves completely the case of surfaces. 

In higher dimension, the assumption that the boundary components are spheres seemed very stringent, the question being whether there exists any example where $M$ is not a topological ball. In a second version of their paper, the authors in \cite{Chen-Erchenko-Gogolev} claim that their argument also applies to the case $\partial M \simeq \mathbb{S}^1 \times \mathbb{S}^{n-2}$. 

The arguments of \cite{Chen-Erchenko-Gogolev} were divided into three parts. The first was to observe that when boundaries are spheres, one can always find a manifold with concave boundary, and constant curvature $-1$ that could play the role of $N\setminus M$ . In the second one, they constructed an extension of the metric of $M$ on a collar near the boundary, that could be glued with a constant curvature metric. In the last part, using a delicate analysis of Jacobi fields, they obtained hyperbolicity of the geodesic flow of the whole closed manifold. 

In this paper, we will give a more general version of the second step in \cite{Chen-Erchenko-Gogolev}, and discuss at length the first step. We will also observe that the arguments in third step are very robust, and extend very generally. \\

Our statement that extends step 2 of \cite{Chen-Erchenko-Gogolev} to a more abstract setting is the following
\begin{theorem}\label{thm:conformally-compact}
Let $M$ be an Anosov manifold with boundary. Then one can embed isometrically $M\subset N$ into a conformally compact complete $n$-manifold $N$ whose interior has uniformly hyperbolic geodesic flow, negative curvature at infinity, asymptotically constant, and no conjugate points. Also, $N$ is diffeomorphic to $\mathring{M}$. \\
\end{theorem}

Let us now turn to our main result concerning step 1. We prove, that indeed, as suspected by several colleagues when \cite{Chen-Erchenko-Gogolev} first appeared as a preprint, simple topology of the boundary is a great constrain on the global geometry of $M$. 
\begin{theorem}\label{thm:special-cases}
Let $M$ be an Anosov manifold with boundary of dimension at least $3$
\begin{itemize} 
	\item If at least one boundary component is diffeomorphic to a sphere, then $M$ is diffeomorphic to a ball, and has no trapped set.
	\item If at least one boundary component is diffeomorphic to a $\mathbb{S}^{n-2}$ bundle over $\Ss^1$, then either the bundle is trivial, and $M$ is diffeomorphic to a solid torus or it is non oriented, and $M$ is diffeomorphic to the product of a M\"obius band and a disk. In both cases, it is a convex neighbourhood of a single closed geodesic. 
\end{itemize}
\end{theorem} 
As observed by the authors of \cite{Chen-Erchenko-Gogolev}, in the second case, one can use the residual finitess of the $\pi_1$ of hyperbolic manifolds to embed $M$ in the finite cover of any given compact hyperbolic $n$-manifold. If one component of the boundary is a $\Ss^{n-p-1}$ bundle over a $p$-dimensional Anosov manifold, one could imagine that $M$ is diffeomorphic to the corresponding ball bundle ; except from some partial results, whether this is true or not has eluded the author.

Finally, we show that for 3 manifolds, the problem can also be completely solved. 
\begin{theorem}\label{thm:3D}
Let $M$ be an oriented Anosov 3-manifold with boundary. Then Problem \ref{probA} has a positive answer. 
\end{theorem}

We also obtain:
\begin{cor}\label{cor:1}
If $(M,g)$ is an oriented Anosov 3-manifold with boundary, there exists a metric $g'$ on $M$, with curvature $-1$, such that $(M,g')$ is also Anosov with boundary, and is a convex neighbourhood of the convex core of a convex co-compact quotient of $\mathbb{H}^3$. 
\end{cor}

If $(N,g)$ is a convex co-compact 3-manifold whose convex core admits a neighbourhood $V$ that has totally geodesic boundary, one can remove the complement of $V$, and take the double of $V$ along its boundary, to obtain a compact manifold with curvature $-1$ that has the convex core of $N$ as a subset. However, there exist convex co-compact 3-manifold for which one cannot find such a $V$, even modifying the metric. For example, the quotient $N=PSL_2(\C)/\Gamma$, where $\Gamma$ is a Schottky group, is the interior of a handlebody, and the double\footnote{to obtain a manifold with negative curvature, one has to glue along the boundary with a map that reverses orientation, not the identity.} of a handlebody does not support a metric of curvature $-1$. 

It seems that current knowledge of topology of Riemannian manifolds of dimension $\geq 4$ is not sufficient to settle the question in higher dimensions. In particular, since there exist compact manifolds of dimension $\geq 4$, that admit negatively curved metrics, but do not admit a structure of homogeneous space, there can be no equivalent of Corollary \ref{cor:1} in higher dimension. 

\textbf{Structure of the arguments.} We will rely on a good part of the argument in \cite{Chen-Erchenko-Gogolev}. Let us recall its gist. They first construct an extension of the metric near the boundary, whose curvature is bounded above uniformly, but becomes arbitrarily negative at an arbitrarily small distance of the boundary, and constant at a fixed distance.
\begin{figure}[h]
\def\svgwidth{0.5\columnwidth}
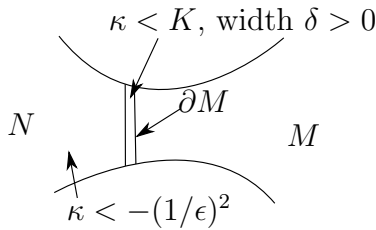
\caption{The extension of the metric near the boundary}
\end{figure}

The second part is to study the dynamics of Jacobi Fields, and prove that the passage of the geodesics through the patch of possibly positive curvature is compensated by the large negative curvature in the rest of the extension, so that the extension does not create conjugate points, and preserves the Axiom A property. 

The last part is to find a manifold of constant sectional curvature that can be glued to the extended manifold. 

We point out\footnote{This observation may be relevant in higher dimension. Indeed, starting from dimension 4, there exist manifolds supporting negatively curved metrics, but supporting no locally symmetric metric.} that in the arguments of the second part, it is not really important that the curvature becomes constant far from the boundary. What is crucial is that it is arbitrarily negative on the complement of an arbitrarily small patch near $\partial M$, and uniformly bounded above. If one were to build an extension satisfying only these two properties, the second part of the arguments would apply. Then if one glues to the extension a manifold with concave boundary and negative curvature, the rest of the arguments also apply. The remaining topological question is thus whether there exist concave manifold with negative curvature and prescribed  metric near boundary components.

\textbf{Organization of the paper.} We will first give a slightly different presentation of the construction of the metric extension, so that it applies to the most general case and prove Theorem \ref{thm:conformally-compact}. Next, we will discuss some topological results about the general problem, and finish the proof of Theorem \ref{thm:special-cases}. Finally, we will concentrate on the 3 dimensional case, and prove Theorem \ref{thm:3D}.

We happily thank C. Lecuire for providing the key argument for Theorem \ref{thm:3D}. We also warmly thank B. Petri and F. Paulin for many explanations. The author of this paper is not an expert in topology of hyperbolic 3 manifold, or higher dimensional topology for that matter. If such an expert would take interest in these questions, maybe significant further progress could be made. 

This work was elaborated with the support of Agence Nationale de la Recherche through the PRC grant ADYCT (ANR-20-CE40-0017).

\section{The extension problem}
\label{sec:metric-extension}

We will here present the constructive arguments of \cite{Chen-Erchenko-Gogolev} a little differently. We start by giving formulæ for the curvatures of a metric in coordinates adapted to an hypersurface. 

\subsection{Study of the curvatures in a slice situation}

In this section, we consider $S$ a compact manifold endowed with a family of Riemannian metrics $(g_t)_t$, and a function $f$ of the parameter $t$, that we will assume to lie in an interval $I\subset \R$. Both are assumed to be smooth, and we are chiefly interested in the sectional curvatures of the metrics on $I\times S$ given by 
\[
h= dt^2 + g_t,\qquad \tilde{h} = dt^2 + f(t)^2 g_t. 
\]
In what follows, we will think of $g_t$ as fixed, and will let $f$ vary to obtain some interesting properties. Quantities related to $\tilde{h}$ will be denoted with a ${}^\sim$. In \cite{Chen-Erchenko-Gogolev}, some very classical formulæ are recalled, but some computations have only been done in local coordinates using Christoffel coefficients; we will here try to give an intrinsic presentation, simplifying the proof of the first part of their statement. 

On the manifold $I\times S$, we denote by $T$ the vector field $\partial/\partial t$, and the letters $X,Y$ will denote vector fields tangent to $S$, that do not depend on $t$. Recall that the second fundamental form is $II =(1/2) \partial_t h$, and the Shape operator $A$ is defined by
\[
h(AX,Y) = II(X,Y). 
\]
Here, 
\[
g_t(\tilde{A}X,Y) = \frac{f'}{f} g_t(X,Y) + \frac{1}{2}\partial_t g_t(X,Y), 
\]
so that
\[
\tilde{A} = \frac{f'}{f} + A.
\]
We will assume that the slices $S_t =\{t\} \times S$ are strictly convex, i.e that $\tilde{A} > 0$. More precisely, we will assume that $A\geq 0$, and that $f'/f > 0$. We say that the slices are \emph{uniformly} strictly convex if we have $\tilde{A}> C >0$ for some global constant $C>0$. 

By $\sigma_{U,V}$, we denote the plane generated by vectors $U,V$.
\begin{proposition}
We have the following formulæ for the sectional curvatures of $\tilde{h}$, where $a\neq 0$, and $X,Y$ are orthonormal for $g_t$ at the point of computation
\begin{align*}
\tilde{K}(\sigma_{X,T}) &= -\frac{f''}{f} + K(\sigma_{X,T}) - 2 \frac{f'}{f} g_t(AX,X).  \\
\tilde{K}(\sigma_{X,Y}) &=  \frac{1}{f^2} K^{int}(\sigma_{X,Y}) + \left[ g_t(AX,Y)^2 - g_t(AX,X)g_t(AY,Y)\right] \\
						&\qquad - \left(\frac{f'}{f}\right)^2  - 2 \frac{f'}{f}\left[ g_t(AX,X)+ g_t(AY,Y)\right]\\
\tilde{K}(\sigma_{X+aT,Y})&= \frac{ f(t)^2 \tilde{K}(\sigma_{X,Y}) + a^2 \tilde{K}(\sigma_{T,Y}) + 2a h( R(X,Y)Y,T)}{f(t)^2 + a^2} .  \\
\end{align*}
\end{proposition}

\begin{proof}
This essentially follows from the Gauss and Codazzi equations (see \S 6.3 in \cite{DoCarmo}). Let us start with the sectional curvature of the plane $\sigma_{X,T}$ generated by $T$ and $X\neq 0$. It is given by
\[
K(\sigma_{X,T}) = - \frac{g_t((A'+A^2)X,X)}{g_t(X,X)},
\]
so that
\[
\tilde{K}(\sigma_{X,T}) = - \frac{f''}{f} + K(\sigma_{X,T}) - 2 \frac{f'}{f} \frac{g_t(AX,X)}{g_t(A,A)}. 
\]
%And we get
%\begin{equation}
%\tilde{K}(\sigma_{X,T}) \leq - \frac{f''}{f} + K(\sigma_{X,T}). 
%\end{equation}

Let us now turn to the curvature of the plane $\sigma_{X,Y}$ generated by $X$ and $Y$. The Gauss' equation gives
\[
K(\sigma_{X,Y}) = K^{int}(\sigma_{X,Y}) - \frac{ g_t(AX,X)g_t(AY,Y) - g_t(AX,Y)^2}{g_t(X,X)g_t(Y,Y) - g_t(X,Y)^2}, 
\]
where $K^{int}$ is the sectional curvature of $g_t$. Since $f$ is constant on each slice, we find directly
\[
\tilde{K}^{int}(\sigma_{X,Y}) = \frac{1}{f^2} K^{int}(\sigma_{X,Y}). 
\]
We deduce that taking $X$ and $Y$ orthonormal (for $g_t$) at the point of interest, we get
\[
\begin{split}
\tilde{K}(\sigma_{X,Y}) &= \frac{1}{f^2} K^{int}(\sigma_{X,Y}) - \left[ g_t(AX,X)g_t(AY,Y) - g_t(AX,Y)^2\right] \\
		&\quad - \left( \frac{f'}{f}\right)^2  - 2 \frac{f'}{f}( g_t(AX,X) + g_t(AY,Y) ). 
\end{split}
\]

Let us now turn to the curvature of the plane $\sigma_{T+aX,Y}$, generated by $X + aT$ and $Y$, for some $a\neq 0$. Without changing the plane (but changing $a$), we can assume that $X,Y$ are orthogonal (for $g_t$) at the point of interest. Then we have 
\[
K(\sigma_{X+aT, Y}) = \frac{ h( R(X+aT, Y)Y, X+aT )}{(a^2+g_t(X,X))g_t(Y,Y)}. 
\]
This gives (using the symmetries of the curvature tensor)
\[
\begin{split}
K(\sigma_{X+aT, Y}) =& \frac{1}{a^2+g_t(X,X)^2}\left(g_t(X,X)^2 K(\sigma_{X,Y}) + a^2 K(\sigma_{T,Y})\right) \\
& + \frac{2a}{(a^2+g_t(X,X)^2)g_t(Y,Y)}\langle R(X,Y)Y,T\rangle. 
\end{split}
\]
It is this last mixed term in the RHS that was estimated using coordinates and Christoffel coefficients in \cite{Chen-Erchenko-Gogolev}. We are seeking to compute
\[
F(X,Y):= \langle \nabla_X \nabla_Y Y - \nabla_Y \nabla_X Y - \nabla_{[X,Y]} Y, T\rangle. 
\]
According to the Codazzi equation, we have
\[
F(X,Y)=\nabla_X^{int} II(Y,Y) - \nabla_Y^{int} II(X,Y). 
\]
Now, we observe that
\[
\tilde{F}(X,Y)= f(t)^2 F(X,Y). 
\]
This $f(t)^2$ term will be compensated by $\tilde{h}(Y,Y)$, so we conclude that
\[
\tilde{K}(\sigma_{X+aT,Y}) = \frac{ f(t)^2 \tilde{K}(\sigma_{X,Y}) + a^2 \tilde{K}(\sigma_{T,Y}) + 2 a F(X,Y) }{f(t)^2 + a^2}.
\]
\end{proof}

\subsection{Extension of the metric}

We assume that we are given $(M,g)$ as in the introduction. Near the boundary, we can always add cylinders, to define
\[
N:= M \cup (\partial M)_x\times ]0, +\infty]_t, 
\]
which is a smooth manifold with boundary, diffeomorphic to $M$. Theorem \ref{thm:conformally-compact} follows from the following
\begin{lemma}\label{lemma:extension}
We can build a metric $\tilde{h}$ on the interior of $N$ that extends the metric of $M$, and so that this metric is conformal to a smooth metric on $N$. Additionally, there exists a constant $C_0>0$ so that for $\kappa>0$ large enough, we can ensure that
\begin{enumerate}
	\item In $\mathring{N}\setminus M$, the sets $\{t=t_0\}$ are equidistant sets from the boundary $\partial M$. They are uniformly strictly convex
	\item The sectional curvature of $\tilde{h}$ tends to $-\kappa^2$, with equality for $t\geq 1$ if $\partial M$ supports a constant curvature metric.
	\item The sectional curvature of $\tilde{h}$ is globally bounded above by $C_0$,
	\item For $t\geq 1/\sqrt{\kappa}$, the sectional curvature of $\tilde{h}$ is less than $-\kappa^2/2$. 
\end{enumerate}
\end{lemma} 

\begin{proof}
We will work near a fixed boundary component $P\subset \partial M$. Near $P$, we can write the metric in the form
\[
h=dt^2 + g_t(x,dx), 
\]
where $g_t$ is a family of metrics on $P$, $-\epsilon< t\leq 0$, and $t$ is the geodesic distance to $P$. We can extend the family $g_t$ smoothly to $]-\epsilon, +\infty[$, and since $\partial_t g_t > 0$ near $ t = 0$, we can ensure that 
\begin{itemize}
	\item $\partial_t g_t > \tfrac{1}{2}\partial_t g_t|_{t=0}$ for $t\in ]-\epsilon, 1/2]$
	\item $\partial_t g_t \geq 0$ for all $t$'s
	\item $\partial_t g_t = 0$ for $t\geq 1$. 
\end{itemize} 
Additionally, if we are given a metric $g'$ on $P$, we can ensure that $g_t = C g'$ for $t\geq 1$ and some constant $C>0$. When $P$ supports a metric of constant sectional curvature, this is the choice that we make. 

Next, we want to look for $\tilde{h} = dt^2 +f(t)^2 g_t(x,dx)$ in the form given by the previous \S. We impose $f=1$ on $]-\epsilon, 0]$.

Let us analyze the conditions of the theorem and their consequence for $f$. First, to ensure uniform strict convexity, it suffices to assume that $f''\geq 0$ globally, and that $f'/f>C>0$ on $[1/2, +\infty[$. Next, $f''\geq 0$ also ensures\footnote{One needs that $f''\geq 0$, $f'\geq 0$ and $f\geq 1$. Since $f=1$ on $]-\epsilon, 0]$ and is smooth, it suffices to assume that $f''\geq 0$} that $\tilde{K}\leq K$, so that (3) is satisfied with $C_0 \geq \sup K$. 

Let us now consider property (2). We observe that since for $t\geq 1$, $\partial_t g_t = 0$, we get the formulæ ($t\geq 1$)
\begin{align*}
\tilde{K}(\sigma_{X,T}) &= -\frac{f''}{f} \\
\tilde{K}(\sigma_{X,Y}) &= \frac{1}{f^2}K^{int}(\sigma_{X,Y}) - \left(\frac{f'}{f}\right)^2\\
\tilde{K}(\sigma_{X+aT,Y})&= \frac{ f(t)^2 \tilde{K}(\sigma_{X,Y}) + a^2 \tilde{K}(\sigma_{Y,T})}{f(t)^2+ a^2}. 
\end{align*}
This suggests to take for $t\geq 1$
\begin{equation}\label{eq:f-away}
f(t)= f_{cc}(t):= \begin{cases} \frac{C}{\kappa} \cosh(\kappa t) & \text{if $\inf K^{int}_{t\geq 1} = - C^2<0$,}\\
e^{\kappa t} &\text{if $K^{int}_{t\geq1} =0$,}\\
\frac{C}{\kappa} \sinh(\kappa t) & \text{if $K^{int}_{t\geq 1}\geq 0$ and $\sup K^{int}_{t\geq 1} = C^2 >0$.}
\end{cases}
\end{equation}
This ensures that the curvature tends to $-\kappa^2$ as $t\to+\infty$ (exponentially fast in $t$). When $g_{t\geq 1}$ has constant curvature, we have $\tilde{K}=-\kappa^2$ for $t\geq 1$. 

Now, for item (4), we observe that for $C_0$ large enough,
\[
\tilde{K}(\sigma_{X,T}) \leq - \frac{f''}{f} + C_0,\quad \tilde{K}(\sigma_{X,Y}) \leq \frac{C_0}{f^2} - (f'/f)^2,
\]
\[
\tilde{K}(\sigma_{X+aT,Y}) \leq \max(\tilde{K}(\sigma_{X,T}),\tilde{K}(\sigma_{X,Y}) ) + \frac{C_0}{f}. 
\]
Let $t_0\in(0,1)$ be such that the function $f_{cc}$ defined in \eqref{eq:f-away} satisfies $f_{cc}\geq 1$ for $t\geq t_0$. Then taking this suggestion for $f$, not only for $t\geq 1$ but $t\geq t_0$, we get
\[
\tilde{K} \leq 2C_0 - \kappa^2.
\]
We observe that (for example), $t_0 = 1/\sqrt{\kappa}$ satisfies the condition for $\kappa$ large. We also observe that $f_{cc}''\geq 0$ for $t\geq 0$, and $f_{cc}'/f_{cc}> C>0$ for $t>1/2$, uniformly as $\kappa$ becomes large. It remains thus to define $f$ on the interval $[0,1/\sqrt{\kappa}]$ so that $f''\geq 0$, to obtain a smooth function. The only condition for this is that $(0,f(0))=(0,1)$ is strictly above the tangent to the graph of $f_{cc}$ at $t=t_0 =1/\sqrt{\kappa}$. Accordingly, we require that
\[
f_{cc}(t_0) - f_{cc}'(t_0) t_0 < 1. 
\]
We can then check case by case that this holds for $t_0 = 1/\sqrt{\kappa}$ and $\kappa$ large enough. 

Let us finally check the smooth conformal compactness. For this, it suffices to set $y = e^{-\kappa t}$, to write the metric in the cylinder in the form
\[
\tilde{h} = \frac{1}{\kappa^2 y^2}\left( dy^2 + m(y)g_{t\geq 1}(x,dx)\right). 
\]
Here, the function $m$ is a smooth function of $y^2$, and $y$ is a boundary defining function. 
\end{proof}

\section{General observations about the topology of the problem}

Let us collect some general informations about universal covers of manifolds with hyperbolic geodesic flow. This is classical material. 

We will denote by $\pi:\widetilde{M}\to M$ the universal cover of $M$, and by $\widetilde{N}$ that of $N$. Since $M$ is geodesically convex, and does not have conjugate points, we can pick $x_0\in M$ and identify $\widetilde{M}$ with the set of $u\in T_x M$ such that $\exp_x(u)\in M\simeq \R^n$. By this construction, we identify $\widetilde{M}$ as a geodesically convex subset of $\widetilde{N}\simeq \R^n$. The fundamental group $\pi_1(M)$ is realized by isometries of $\widetilde{N}$, which preserve $\partial \widetilde{M}$. 

In particular, $\partial\widetilde{M}$ is a Galois cover of $\partial M$. The embedding $\imath : \partial M \hookrightarrow M$ induces a map $\imath_\ast : \pi_1(\partial M)\to \pi_1(M)$. If $P$ is a connected component of $\partial M$, and $\widetilde{P}$ a connected component of $\pi^{-1}(P)$, then we have $\pi_1(\widetilde{P}) = \{ \gamma \in \pi_1(P)\ |\ \imath_\ast(\gamma) = 0\}$, and $P = \widetilde{P}/\imath_\ast(\pi_1(P))$. 

The elements of $\pi_1(M)$ (except $1$) are represented by a special kind of isometries, called \emph{loxodromic}. They have a continuous extension to (H\"older) homeomorphisms of $B^n = \widetilde{N}\cup\Ss^{n-1}$ (the \emph{visual compactification} of $\widetilde{M}$). They have exactly two fixed points, which lie in $\Ss^{n-1}$. They preserve the corresponding geodesic, along which they are a translation. In particular, if two isometries $\gamma$, $\mu$ commute, there must exist another $\eta$, and $k,\ell\in\N$ so that $\gamma = \eta^k$, $\mu = \eta^\ell$. This implies that there is no copy of $\Z^2$ inside $\pi_1(M)$. 

Projecting $\partial \widetilde{M}$ along rays from $x_0$ on $\Ss^{n-1}$, we find that $\partial\widetilde{M}$ is (H\"older) homeomorphic to an open set of $\Ss^{n-1}$. Its complement $\Lambda$ is called the \emph{limit set}. It is the set of endpoints of geodesics that remain inside in $M$ for all times.

From these general facts, we deduce our Theorem \ref{thm:special-cases}. 
\begin{proof}
Let us start by observing that the following are equivalent
\begin{enumerate}
	\item At least one connected component of $\partial \widetilde{M}$ is a sphere
	\item At least one connected component of $\partial \widetilde{M}$ is compact
	\item $\partial\widetilde{M}\simeq \Ss^{n-1}$
	\item $\widetilde{M}\simeq B^n$
	\item $\pi_1(M) =\{1\}$
	\item $M$ is diffeomorphic to a closed ball. 
\end{enumerate}

Certainly, (1) implies (2), and since $\partial \widetilde{M}$ is an open set of $\Ss^{n-1}$, if one of its connected component is compact, it must the whole sphere, so (2) implies (3). Using rays starting from $x_0$, we can then build a diffeomorphism between $\widetilde{M}$ and $B^n$, so that (3) implies (4). In that case, every element of $\pi_1(M)$ preserves a compact set of $\R^n$, so must be trivial, and (4) implies (5). If (5) holds, then $M= \widetilde{M}$ must be compact, and using again rays from $x_0$, we find that $M$ is a closed ball. Finally, if $M$ is a closed ball, its boundary is a sphere. This takes care of item (1) of Theorem \ref{thm:special-cases}. \\

Now, we also have equivalence between
\begin{enumerate}[(a)]
	\item $\Lambda$ has exactly two points
	\item The boundary of $M$ is diffeomorphic to a bundle: $\Ss^{n-2} \to \partial M \to \Ss^{1}$. 
	\item One connected component of $\partial M$ is diffeomorphic to a bundle $\Ss^{n-2} \to P \to \Ss^{1}$. 
	\item One connected component $P$ of $\partial M$ satisfies $\imath_\ast(\pi_1(P))\simeq \Z$. 
\end{enumerate}

Certainly, (a) implies that $\pi_1(M)$ is generated by a non trivial loxodromic isometry $\gamma$. Taking Fermi-coordinates along the geodesic it preserves, this gives a decomposition of $\tilde{N}$ in the form $\R\times \R^{n-1}$, where $\gamma$ acts as follows:
\[
\gamma(t,x) = (t+\ell_\gamma, A_\gamma x).
\]
Here, $\ell_\gamma$ is the translation length of $\gamma$ and $A_\gamma\in O(n-1)$. The preserved geodesic is given by $\R\times\{0\}$. Since $\gamma$ must be represented in $M$ by a closed geodesic, this implies that $\R \times \{0\} \subset \widetilde{M}$. The convexity of $M$ implies that the intersection $\widetilde{M}\cap \{t\} \times \R^{n-1}$ is star-shaped for every $t\in\R$. Additionally, since $\pi_1(M) = <\gamma>$, this intersection must be compact, with smooth boundary, so that it is diffeomorphic to a ball. We deduce that the boundary of $M$ is diffeomorphic to the sphere bundle
\[
\R \times \Ss^{n-1}/{(t,x)\sim (t+\ell_\gamma, A x)}. 
\]
Topologically, there are exactly two such bundles\footnote{Because $SO(n-1)$ is connected.}: either $A\in SO(n-1)$, and this is the trivial bundle, or $\det A = -1$, and $M$ is not oriented, with a M\"obius-band like structure. This implies (b) (which implies (c)). 

Now assume (c) and let $P$ be the corresponding connected component, and $\widetilde{P}$ a connected component of $\pi^{-1}(P)$. From the arguments in the ball case above, we know that $\widetilde{P}$ cannot be compact. If $n>3$, this implies $\widetilde{P} = \R \times \Ss^{n-2}$, and $\imath_\ast(\pi_1(P)) \simeq \Z$ is generated by one non trivial isometry. In the case $n=3$, we have to consider that case $\widetilde{P}\simeq \R^2$. However, since $\pi_1(M)$ cannot contain $\Z^2$, this case is ruled out, and so it is the same as $n>3$, and we have (d). 

Now, assuming (d), let $\imath_\ast(\pi_1(P))=<\gamma>$. Then let $c(t)$ be the geodesic preserved by $\gamma$. We find that $\widetilde{P}$ is at bounded distance from $\{c(t) | {t\in\R}\}$. This implies that $\partial \widetilde{P}$ (seen as a subset of $\Ss^{n-1}$) can only contain the endpoints of $c(t)$, and so this implies (a).  %In the case that the bundle is trivial, this gives item (b) of Theorem \ref{thm:special-cases}.
\end{proof}

We have identified two type of particularly simple boundaries. Let us discuss now some more partial results in this direction. 
\begin{lemma}
Let us assume that a boundary component $P$ satisfies $\Gamma:=\imath(\pi_1(P))= \pi_1(\Sigma)$, where $\Sigma$ is a compact Anosov manifold of dimension $p\leq n-2$. Then $P$ is the only boundary component. 
\end{lemma}

\begin{proof} In that case, the visual boundary $\partial \Gamma$ is homeomorphic to a sphere $\Ss^{p-1}$, and since $\Gamma$ is a hyperbolic subgroup of $\pi_1(M)$, the limit set of $\Gamma$, which must be also $\partial \tilde{P}$, is also homeomorphic to $\Ss^{p-1}$. Now, since $\Ss^{n-1}\setminus\Ss^{p-1}$ must be connected, we deduce that $\Ss^{n-1}\setminus \Ss^{p-1} = \tilde{P}$. This means that $P$ is the whole boundary of $M$. 
\end{proof}

If the embedding of $\Ss^{p-1}$ into $\Ss^{n-1}$ is the standard one, we get that $\tilde{P}\simeq \Ss^{n-p-1}\times\R^{p}$, and $M$ turns out to be tubular neighbourhood of $\Sigma$. However, there are more than one way to embed a sphere in another, except in the case that $n> 2p$. 
\begin{cor}
If $\Sigma$ is a surface, and $M$ has dimension at least $5$, then $M$ is diffeomorphic to a ball bundle over $\Sigma$.
\end{cor}

In 4 dimensions, understanding exactly which $B_2$ bundles over a compact surface can be endowed with a complete convex hyperbolic structure is not completely solved. This question was investigated by several authors --- see \cite{gromov}, \cite{convex-plumbing}.

Let us close this section with a modicum of information regarding the general case for the possible shape of the boundary. We specialize to the case of $M$ having 4 dimensions, to be able to use the geometrization theorem. For this, we will rely on several results from the theory of the topology of 3 manifolds. For example, the reader can consult \cite{Freitas}, in particular its \S 3. Since it is not our main focus, we will be very cursory.
\begin{lemma}
Let $M$ be an Anosov $4$-manifold with boundary, with some boundary component $P\neq \emptyset$. Assume that $P\subset M$ is $\pi_1$-injective. Then either $P=\Ss^3$ or it decomposes as a connected sum
\[
A_1 \# ... \# A_p \# B_1 \# \dots \# B_\ell, 
\]
where each $A_j$ is a copy of $\Ss^2\times \Ss^1$, and each $B_j$ is a compact hyperbolic 3 manifold. 
\end{lemma}

\begin{proof}
According to the prime decomposition theorem, we can decompose
\[
P = P_1 \# \dots \# P_m,
\]
where no $P_j$ can be decomposed as a non-trivial connected sum, and (unless $P= \Ss^3$) each $P_j$ is either $\Ss^1\times\Ss^2$ or is irreducible (i.e every embedding of $\Ss^2$ bounds a ball). According to Van-Kampen's theorem, we have that
\[
\pi_1(P)= \pi_1(P_1)\ast \dots \ast \pi_1(P_m).
\]

Now, according to the geometrization theorem, we can decompose each irreducible $P_j$ as
\[
P_j = \cup Q_{j,\ell}, 
\]
where each $Q_{j,\ell}$ is a manifold whose boundary components are torii, and whose interior admits a geometric structure, in the list of Thurston's 8 geometries. Also, each torus boundary is $\pi_1$ embedded. 

Now, since we assumed that $P$ is $\pi_1$-injective, and $\pi_1(M)$ cannot contain a $\Z^2$ subgroup, we deduce that there cannot exist any incompressible torus in the boundary, and, in particular, the decomposition of the $P_j$ must be trivial: each of them supports a Thurston geometry. 

Now, among the Thurston geometries, we observe that for either a compact Euclidean, Nil or Sol manifold, the fundamental group must be a semi-direct product $\Z^2 \times \Z$, which contains a $\Z^2$. 

In the $\Hh^2\times\R$ or $\widetilde{SL(2,\R)}$ case, the fundamental group must contain a cyclic normal subgroup, i.e a $\Z$ center. Since we cannot have torsion, it must also contain a $\Z^2$, and this is also ruled out. 

Next, we observe that in the spherical case, elements of the $\pi_1$ must have finite order, which is not possible because there are no elliptic elements in $\pi_1(M)$. In particular, if $P_j$ has spherical geometry, $P_j = \Ss^3$.

We deduce that each $P_j$ has geometry either $\Ss^2\times\R$, or $\Hh^3$, or is a sphere.
\end{proof}

\section{The case of 3-manifolds}

In this section, let us concentrate on the case that $M$ is 3 dimensional and oriented. Then the boundary components are compact oriented surfaces, so either a sphere, a torus or surfaces of genus $g>1$. As we have seen, if $M$ is neither a solid torus nor a ball, all the components must be of the latter variety. The authors of \cite{Chen-Erchenko-Gogolev} have already commented on how to embed the torus or the ball, so we will concentrate on the remaining case. As noted in \S\ref{sec:metric-extension} we can assume that the boundary components have constant curvature. 
\begin{theorem}\label{thm:compact-embedding-geometry}
Let $M$ be a 3-manifold whose boundary is strictly convex, with curvature $-\kappa^2$ constant near the boundary. Assume further that all connected components of the boundary are hyperbolic surfaces of genus $g>1$. Provided the hyperbolic metric is well chosen on the boundary, $M$ can be embedded into a compact manifold $N'$ without boundary, such that the curvature is $-\kappa^2$ in $N'\setminus M$. 
\end{theorem}

\begin{proof}
We start by recalling this result (see Theorem 3.3 in \cite{Fujii}).
\begin{lemma}\label{lemma:Fujii}
For $g>1$, there exists $N_g$ a compact hyperbolic 3-manifold whose boundary is a totally geodesic surface $S_g$ of genus $g$. 
\end{lemma}

(As explained in \cite{Fujii}, not all hyperbolic structures on surfaces can be realized as totally geodesic boundaries, essentially because of Mostow rigidity). 

Let us come back to our problem. We will from now on assume that the boundary of $M$ is connected. If there are several components, one can work componentwise. Let us hence denote $\Sigma=\partial M$, and endow $\Sigma$ with a hyperbolic metric $g_{\partial N_g}$ so that Lemma \ref{lemma:Fujii} applies, and we also have $\Sigma = \partial N_g$ as a totally geodesic boundary. Near the boundary of $(N_g, h_0)$, we have
\begin{equation}\label{eq:near-boundary-N_g}
N_g \simeq \Sigma_x \times [0, \delta[_\tau,\quad h_0 = d\tau^2 + \cosh(\tau)^2 g_{\partial N_g}(x,dx).
\end{equation}
Let us apply the construction of \S\ref{sec:metric-extension}. For $C_0>0$ large enough, we can ensure that $C_0 g_{\partial N_g}> g_{\partial M}$, where $g_{\partial M}$ is the metric on $\Sigma\simeq \partial M$. We can thus build an extension $M\subset N$ according to Lemma \ref{lemma:extension}. In $N\setminus M$, for $1\leq t\leq 2$, the metric takes the form (for some $\kappa>0$)
\[
dt^2 + \frac{1}{\kappa^2}\cosh(\kappa t)^2 g_{\partial N_g}(x,dx). 
\] 
In the formula above we recognize the expression in coordinates of the metric $h_0/\kappa^2$ (here, $\tau = \kappa t$). If the local coordinates in \eqref{eq:near-boundary-N_g} extend as far as $\tau \leq 2\kappa$, we can thus glue $N_{\{t< 2\}}$ with $N_g\setminus \Sigma\times[0,2\kappa[$, and this will conclude the proof of Theorem \ref{thm:compact-embedding-geometry} setting 
\[
N' = N_{\{t<2\}}\cup (N_g\setminus \Sigma\times [0,2\kappa[).
\]
The difficulty here is that for the dynamical arguments of \cite{Chen-Erchenko-Gogolev} to apply, we need to be able to take $\kappa$ arbitrarily large. We will thus be done if we can prove
\begin{lemma}\label{lemma:good-N_g}
For any $\kappa>0$, we can choose $N_g$ such that the $2\kappa$ neighbourhood of $\partial N_g$ is diffeomorphic to $\partial N_g \times [0,2\kappa[$. 
\end{lemma}

The argument of the proof was communicated by C. Lecuire. It relies on the following very fine statement from the topology of hyperbolic manifolds
\begin{theorem}[Theorem 9.2, \cite{Agol2012}]
Fundamental groups of compact 3 dimensional hyperbolic manifolds are LERF (locally extended residually finite).
\end{theorem} 

This means that whenever $H\subset \pi_1(M)$ is finitely generated, and $\gamma\in \pi_1(M)\setminus H$, there exists a finite index subgroup $\Gamma\subset\pi_1(M)$ such that $H\subset \Gamma$ and $\gamma\notin \Gamma$.

Let us come back to 
\begin{proof}[Proof of Lemma \ref{lemma:good-N_g}]
We will build $N_g$ by induction, so we start by denoting $N_g^0$ the one provided by Lemma \ref{lemma:Fujii}. Let us now consider geodesics of $N_g^0$ with endpoints in its boundary. We will say that two such geodesics are boundary-homotopic, if there is a free homotopy between them, so that at each time, the endpoints remain in the boundary. Working on the universal cover, it is elementary to observe that any such geodesic is either boundary homotopic to a point in the boundary, or to a geodesic that intersects the boundary orthogonally. Additionally, such a boundary-orthogonal geodesic uniquely minimizes the length among its boundary-homotopy class. Let us set
\[
\mathfrak{G}(N_g^0)= \{ a \ |\  \text{some boundary-orthogonal geodesic has length $a$.} \}
\]
Certainly, $\mathfrak{G}(N_g^0)$ is a discrete set, and its only accumulation point is $+\infty$. We can enumerate it as
\[
\mathfrak{G}(N_g^0)= \{ 2 \ell_j \ |\ j\geq 0\}. 
\]

With a bit more of elementary Riemannian geometry, we find that $\ell_0$, the largest $\ell$ such that the $\ell$-neighbourhood of $\Sigma$ is product, is exactly half the length of a shortest boundary-orthogonal geodesic (there may be several non-boundary homotopic orthogonal geodesics with the same length).

Let us denote $m_0>0$ the number of boundary-orthogonal geodesics with length $2\ell_0$, and $\gamma_0$ one of them. Let us now consider the double $A_0 = DN_g^0$ of $N_g^0$ along $\Sigma$, and $D\gamma_0$ the double of $\gamma_0$, a closed geodesic of length $4\ell_0$. Using the LERF theorem, we find a finite index subgroup $\Gamma\subset \pi_1(A_0)$ so that $\pi_1(\Sigma) \subset \Gamma$, and $[D\gamma_0]\notin \Gamma$. Let us denote $B_0 = \Hh^3/\Gamma$ the corresponding cover of $A_0$. 

Removing from $B_0$ all the lifts of $\Sigma$, we obtain a disjoint union of $T_j$'s, so that each $T_j$ is a finite cover of $N_g^0$. At least one of them has more than one boundary component, otherwise we would have $B_0 = A_0$. For every $T_j$, we can consider the lifts of $\gamma_0$ to $T_j$, and in particular the ones that have endpoints in different boundary components. If there is no such lift, $D\gamma_0$ would be lifted to $B_0$ as a closed geodesic, which is a contradiction, so that we can find at least one $T_{j_0}$, and one lift as desired, which starts from some boundary component $\Sigma_0 \simeq \Sigma$. Let us consider $N_g^{0,1}$ the manifold obtained by glueing to $T_{j_0}$ copies of $N_g^0$ along the boundary components different from $\Sigma_0$. 

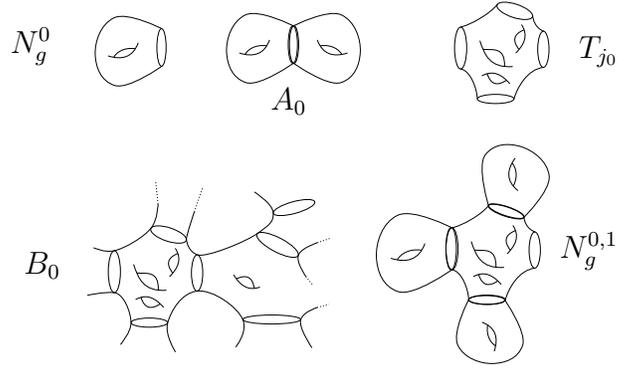
\begin{figure}[h]
\def\svgwidth{0.7\columnwidth}
%% Creator: Inkscape 1.1.2 (0a00cf5339, 2022-02-04), www.inkscape.org
%% Accompanies image file 'BuildingNg.pdf' (pdf, eps, ps)
\begingroup%
  \makeatletter%
  \providecommand\color[2][]{%
    \errmessage{(Inkscape) Color is used for the text in Inkscape, but the package 'color.sty' is not loaded}%
    \renewcommand\color[2][]{}%
  }%
  \providecommand\transparent[1]{%
    \errmessage{(Inkscape) Transparency is used (non-zero) for the text in Inkscape, but the package 'transparent.sty' is not loaded}%
    \renewcommand\transparent[1]{}%
  }%
  \providecommand\rotatebox[2]{#2}%
  \newcommand*\fsize{\dimexpr\f@size pt\relax}%
  \newcommand*\lineheight[1]{\fontsize{\fsize}{#1\fsize}\selectfont}%
  \ifx\svgwidth\undefined%
    \setlength{\unitlength}{381.25643683bp}%
    \ifx\svgscale\undefined%
      \relax%
    \else%
      \setlength{\unitlength}{\unitlength * \real{\svgscale}}%
    \fi%
  \else%
    \setlength{\unitlength}{\svgwidth}%
  \fi%
  \global\let\svgwidth\undefined%
  \global\let\svgscale\undefined%
  \makeatother%
  \begin{picture}(1,0.54431838)%
    \lineheight{1}%
    \setlength\tabcolsep{0pt}%
    \put(0,0){\includegraphics[width=\unitlength,page=1]{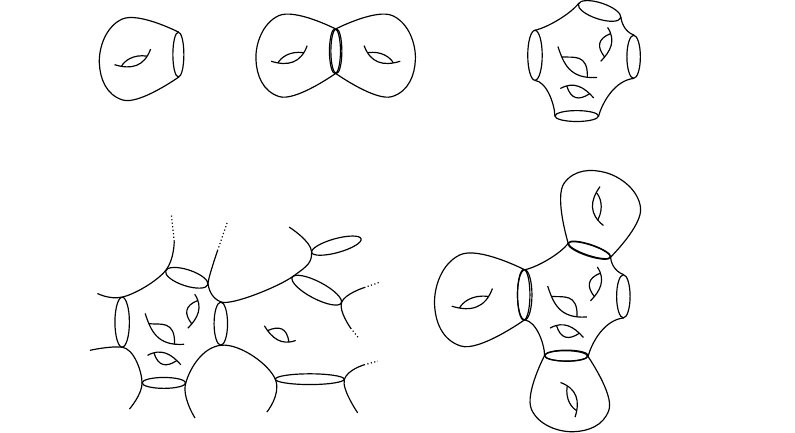}}%
    \put(-0.00217721,0.4672829){\color[rgb]{0,0,0}\makebox(0,0)[lt]{\lineheight{1.25}\smash{\begin{tabular}[t]{l}$N_g^0$\end{tabular}}}}%
    \put(0.38865061,0.38174585){\color[rgb]{0,0,0}\makebox(0,0)[lt]{\lineheight{1.25}\smash{\begin{tabular}[t]{l}$A_0$\end{tabular}}}}%
    \put(0.85129839,0.46021649){\color[rgb]{0,0,0}\makebox(0,0)[lt]{\lineheight{1.25}\smash{\begin{tabular}[t]{l}$T_{j_0}$\end{tabular}}}}%
    \put(0.01848289,0.13513524){\color[rgb]{0,0,0}\makebox(0,0)[lt]{\lineheight{1.25}\smash{\begin{tabular}[t]{l}$B_0$\end{tabular}}}}%
    \put(0.82235226,0.15997178){\color[rgb]{0,0,0}\makebox(0,0)[lt]{\lineheight{1.25}\smash{\begin{tabular}[t]{l}$N_g^{0,1}$\end{tabular}}}}%
  \end{picture}%
\endgroup%

\caption{Construction of $N_g^{0,1}$.}
\end{figure}

If $\gamma$ is a geodesic of $N_g^{0,1}$, with endpoints in the boundary, and orthogonal to the boundary, there is a dichotomy. Either it remains inside $T_{j_0}$, and has a length in $\mathfrak{G}(N_g^0)$, because $T_{j_0}$ is a cover of $N_g^0$. In this case, it cannot be a lift of $\gamma_0$ (because of our choice of $T_{j_0}$). Or it must reach the other boundary components of $T_{j_0}$. In this case, its length must thus be at least $4\ell_0$. 

We deduce that $N_g^{0,1}$ satisfies the same interesting properties as $N_g^0$ (i.e is hyperbolic, and has totally geodesic boundary isometric to $\Sigma$), and
\[
\mathfrak{G}(N_g^{0,1}) \subset \mathfrak{G}(N_g^0)\cup [4\ell_0,+\infty[.
\]
Additionally, there are at most $m_0-1$ geodesics orthogonal to its boundary with length exactly $2\ell_0$.

By finite descent, we build thus a manifold $N_g^1$ with the same interesting properties as $N_g^0$, such that
\[
\mathfrak{G}(N_g^{1}) \subset (\mathfrak{G}(N_g^0)\setminus\{2\ell_0\})\cup [4\ell_0,+\infty[.
\]
Repeating the argument for $\ell_1$ (if $\ell_1<4\ell_0$), we build in a finite number of steps a manifold $N_g^2$ with same interesting properties as $N_g^0$, such that 
\[
\mathfrak{G}(N_g^{2}) \subset [4\ell_0,+\infty[.
\]
By finite induction, we finally obtain a manifold $N_g^3$ with same interesting properties as $N_g^0$ such that 
\[
\mathfrak{G}(N_g^{3}) \subset ]4\kappa,+\infty[.
\]
This is what we wanted to find.
\end{proof}
This closes the proof of Theorem \ref{thm:compact-embedding-geometry}
\end{proof}

Observe that there are examples of compact hyperbolic 4-manifolds whose $\pi_1$ is not LERF. One of the incarnations of the fact that our problem here becomes much more difficult when the dimension increases. 

Let us prove Corollary \ref{cor:1}:
\begin{proof}[Proof of Corollary \ref{cor:1}]
Let $(M,g)$ be an oriented Anosov 3-manifold with boundary. Then we have seen that $M$ can be embedded as a convex open set of a compact manifold $(M',g)$, with Anosov geodesic flow. According to the geometrization theorem, since $M'$ is aspherical, it must be irreducible. Additionally, there can be no $\Z^2$ in its fundamental group, so that it is actually locally homogeneous, and it must be a hyperbolic manifold, with metric $\tilde{g}$. We deduce that $M$ supports a metric with constant curvature $-1$, and $\Gamma=\pi_1(M)\subset PSL_2(\C)$.

On the other hand, we can find a conjugation between the geodesic flows. Working first on the universal cover, given a vector $(x,v)\in S^\ast \Hh^3$, we can consider the endpoint $p_{\pm}$ of the geodesic through $(x,v)$ for the metric $g$. Then if ${\gamma}$ is the geodesic for $\tilde{g}$ with the same endpoints, and $H$ the horosphere based at $p_+$ for the metric $\tilde{g}$, through $x$. Then $H$ and $\gamma$ intersect at exactly one point $\gamma(t)$, and we set $h(x,v) = (\gamma(t),\gamma'(t))$. This map is uniformly H\"older, and equivariant under the action of $\pi_1(M')$, so it defines a continuous map of $S^\ast (\Hh^3/\Gamma)$ to itself, mapping closed geodesics of $g$ to closed geodesics of $\tilde{g}$. 

Let $\Lambda$ be the limit set of $\Gamma$, and let $F$ be the union of geodesics for $g$, with endpoints in $\Lambda$. By definition, $F/\Gamma$ is contained in $M$ and closed, so it is compact. Now, denote by $F_0$ the union of geodesics for the hyperbolic metric on $\Hh^3$ with endpoints in $\Lambda$. Then using the map $h$, we see that the compactness of $F$ implies the compactness of $F_0$, so that $\Gamma$ must be convex-cocompact. 
\end{proof}

Let us conclude the proof of the main theorem \ref{thm:3D}. 
\begin{proof}[Proof of Theorem \ref{thm:3D}]
We have already embedded isometrically $M\subset N'$, so that the curvature of $N'$ satisfies
\begin{enumerate}
	\item $K \leq C_0$ globally,
	\item $K \leq - \kappa^2 /2$ on $N'\setminus V$, where $V$ is the $1/\sqrt{\kappa}$ neighbourhood of $M$. 
\end{enumerate}
We have seen that $C_0$ is fixed, and $\kappa$ can be taken arbitrarily large. We have also ensured that the slices are uniformly strictly convex. In the arguments below, $C>0$ will denote a constant that does not depend on the choice $\kappa$, and that may change at every line. Except from this, we will use the notations for Jacobi fields introduced in \cite{Chen-Erchenko-Gogolev}. Let us just recall that if $J$ is a Jacobi field, 
\[
\mu_J = \frac{1}{2}\frac{(\| J \|^2)'}{\|J\|^2}.
\]
We also denote by $\partial_{\pm} SM$ the set of vectors in $SM$ above $\partial M$ that are entering (-) and outgoing (+). Also denote by $\Gamma_{\pm}$ the set of entering (-) (resp. outgoing (+)) vectors that are trapped in $M$ for all positive (resp. negative) time. 

Observe that the constant $1/C_1$ in \cite{Chen-Erchenko-Gogolev} corresponds to our $1/\sqrt{\kappa}$. The key technical tool is the comparison lemma 2.8 in \cite{Chen-Erchenko-Gogolev}. In particular, it tells us that in $N'\setminus V$, in the region $t>1/\sqrt{\kappa}$, where the curvature is below $-\kappa^2/2$, a Jacobi field with $\mu_J(0)> - Q$ satisfies $\mu_J(t)>(1-\epsilon)\kappa/\sqrt{2}$ for $t\gtrsim 1$, when $\kappa$ is large and $Q$ remains fixed.

We will rely on their arguments of \S8, trying to give some detail. Let us first sketch the proof of
\begin{lemma}
Let $v\in \partial_- SM$ and $J$ be a perpendicular Jacobi field along the corresponding geodesic. If $\mu_J(0)>Q_M$, then $J$ does not vanish in positive time. 
\end{lemma}

\begin{proof}
According to Lemma 8.5 in \cite{Chen-Erchenko-Gogolev}, the time travel in the region $t\in [0,1/\sqrt{\kappa}]$ is bounded above by $C/\sqrt{\kappa}$. If $J$ a perpendicular Jacobi field along a geodesic $\gamma_v$ starting at a point $v\in SN'$, let us assume that $v\in \partial SM$ is entering. Then according to Proposition 4.1 of \cite{Chen-Erchenko-Gogolev}, if $\mu_J(0) > Q_M$, either $v\in\Gamma_-$ and we are done, or $\gamma\notin \Gamma_-$. Then as $\gamma_v$ leaves $SM$, $\mu_J(\ell(v)) > - Q_M$. Since the travel time in the collar is small and the curvature bounded above, we deduce that $\mu_J$ is at least $-2Q_M$ when $\gamma_v$ enters $\{t>1/\sqrt{\kappa}\}$. Now, in this region, the curvature is bounded above by $-\kappa^2/2$, so that applying the comparison lemma 2.8 of \cite{Chen-Erchenko-Gogolev}, and taking into account that the travel time inside $t>1/\sqrt{\kappa}$ must be at least $1$, if $\kappa$ is large enough, $\mu_J$ must be at least $Q_M$ again when $\gamma_v$ enters again $SM$ (if it ever does). We can conclude by induction on the times. 
\end{proof}

\begin{lemma}
$N'$ has no conjugate points
\end{lemma}

\begin{proof}
Let us consider a nonzero Jacobi field $J$ that vanishes at some point above $N'\setminus M$. Then, using again the comparison lemma, we deduce that $\mu_J> Q_M$ as long as the geodesic remains in $SN'\setminus SM$, and thus cannot vanish again according to our lemma.

Let us now consider a nonzero Jacobi field vanishing at some point. If the corresponding geodesic remains in $SM$, then we are done. Otherwise consider the first time that it exits $SM$. Then, using the argument of the proof of Lemma 8.11 of \cite{Chen-Erchenko-Gogolev}, we must have $\mu_J> - Q_M$. Otherwise we could apply Proposition 4.1 in reversed time, and obtain a contradiction. Again by the smallness of the collar and the very negative curvature in $\{t>1/\sqrt{\kappa}\}$, if the geodesics come back again in $SM$, we must have $\mu_J> Q_M$ at the point of entry. Then proceeding by induction as above enables us to conclude. 
\end{proof}

Finally, we have to prove that the geodesic flow is Anosov. For this it suffices, according to Eberlein's theorem, to prove that no nonzero Jacobi field can be globally bounded. Since $M$ has axiom A geodesic flow and the curvature is very negative in $\{t>1/\sqrt{\kappa}\}$, this is a given for any Jacobi field along a geodesic that remains either in $SM$ or $SN'\setminus SM$ for all times.

Consider a geodesic $\gamma$ entering $SM$ at a point $v$, with a perpendicular Jacobi field $J$ satisfying $\mu_J> Q_M$ at that point. Then according to Proposition 4.1 of \cite{Chen-Erchenko-Gogolev}, we have either $\|J\|\to+\infty$, or $\int_{SM} \mu_J > -C_0$ before $\gamma$ exits $SM$. Now, before the geodesic enters again (if it ever does) in $SM$, since the curvature is so negative, using the comparison lemma again, we must have 
\[
\int_{SN\setminus SM} \mu_J\geq (1-\epsilon) \frac{\kappa}{\sqrt{2}} - \epsilon,
\]
where $\epsilon$ tends to $0$ as $\kappa$ grows large. Here we have taken into account that the travel time must be at least $1$ above $N'\setminus M$. We also used Proposition 4.1, ensuring that $\mu_J>-Q_M$ entering $SN'\setminus SM$. 

Now, either the sequence of times when the geodesic changes component is finite, and we know directly by the comparison lemma and Proposition 4.1 that $|J|\to+\infty$, or the sequence is infinite. In that case, we observe that between two consecutive times the geodesic entered $SM$, we have
\[
\int \mu_J > \frac{1}{10}\kappa - C_0 > 1, 
\]
so that $\|J\|\to+\infty$ also in positive time. 

Let us now consider the case that as $\gamma$ enters $SM$, $\mu_J \leq Q_M$. Then, we reverse time, and see this as a geodesic entering $SN'\setminus SM$ with $\mu_J \geq - Q_M$. Then the same argument as above apply, only in negative time.
\end{proof}

\bibliographystyle{alpha}
\bibliography{biblio}

\begin{thebibliography}{AGM12}

\bibitem[AGM12]{Agol2012}
Ian Agol, Daniel Groves, and Jason Manning.
\newblock The virtual haken conjecture, 2012.

\bibitem[CEG23]{Chen-Erchenko-Gogolev}
Dong Chen, Alena Erchenko, and Andrey Gogolev.
\newblock Riemannian anosov extension and applications, 2023.

\bibitem[dC15]{DoCarmo}
Manfredo~Perdig{\~a}o do~Carmo.
\newblock {\em Geometria {Riemanniana}}.
\newblock Projeto Euclides. Rio de Janeiro: Instituto de Matem{\'a}tica Pura e
  Aplicada (IMPA), 5th edition, 2nd printing edition, 2015.

\bibitem[dFR22]{Freitas}
Izabella~Muraro de~Freitas and \'Alvaro~Kr\"uger Ramos.
\newblock Geometrization in geometry, 2022.

\bibitem[Fuj90]{Fujii}
Michihiko Fujii.
\newblock Hyperbolic 3-manifolds with totally geodesic boundary.
\newblock {\em Osaka J. Math.}, 27(3):539--553, 1990.

\bibitem[GLT88]{gromov}
M.~Gromov, H.~Blaine~jun. Lawson, and W.~Thurston.
\newblock Hyperbolic 4-manifolds and conformally flat 3-manifolds.
\newblock {\em Publ. Math., Inst. Hautes {\'E}tud. Sci.}, 68:27--45, 1988.

\bibitem[MRS21]{convex-plumbing}
Bruno Martelli, Stefano Riolo, and Leone Slavich.
\newblock Convex plumbings in closed hyperbolic 4-manifolds.
\newblock {\em Geom. Dedicata}, 212:243--259, 2021.

\bibitem[Rug91]{Ruggiero-1991}
Rafael~Oswaldo Ruggiero.
\newblock On the creation of conjugate points.
\newblock {\em Math. Z.}, 208(1):41--55, 1991.

\end{thebibliography}
\end{document}